\documentclass{amsart}
\usepackage{mathrsfs, 
	amsmath, 
	amssymb, 
	amsthm, 
	amscd, 
	amstext, 
	verbatim, 
	enumerate, 
	aliascnt, %
	color}
\usepackage{enumitem}
\setlist[enumerate,1]{label={(\roman*)}} %
\usepackage{todonotes}
\usepackage{hyperref}
\hypersetup{				%
	colorlinks,
	linkcolor={blue},
	citecolor={blue},
	urlcolor={blue}
}
\usepackage{mathtools}
\usepackage[all]{xy}
\allowdisplaybreaks[4]

\theoremstyle{plain}

\newtheorem{thm}{Theorem}[section]

\newaliascnt{propCt}{thm}
\newtheorem{prop}[propCt]{Proposition}
\aliascntresetthe{propCt}

\newaliascnt{lemCt}{thm}
\newtheorem{lem}[lemCt]{Lemma}
\aliascntresetthe{lemCt}

\newaliascnt{corCt}{thm}
\newtheorem{cor}[corCt]{Corollary}
\aliascntresetthe{corCt}

\newaliascnt{extCt}{thm}

\aliascntresetthe{extCt}

\newaliascnt{algoCt}{thm}

\aliascntresetthe{algoCt}

\newaliascnt{notaCt}{thm}

\aliascntresetthe{notaCt}

\newaliascnt{conjCt}{thm}

\aliascntresetthe{conjCt}

\newtheorem{introthm}{Theorem}

\newaliascnt{introcorCt}{introthm}
\newtheorem{introcor}[introcorCt]{Corollary}

\newaliascnt{introdefCt}{introthm}

\aliascntresetthe{introdefCt}

\theoremstyle{definition}

\newaliascnt{defnCt}{thm}
\newtheorem{defn}[defnCt]{Definition}
\aliascntresetthe{defnCt}

\newaliascnt{convCt}{thm}
\newtheorem{conv}[convCt]{Convention}
\aliascntresetthe{convCt}

\newaliascnt{remCt}{thm}
\newtheorem{rem}[remCt]{Remark}
\aliascntresetthe{remCt}

\newaliascnt{egCt}{thm}
\newtheorem{eg}[egCt]{Example}
\aliascntresetthe{egCt}

\newaliascnt{assuCt}{thm}

\aliascntresetthe{assuCt}

		\newcommand{\C}{\mathbb C}

		\DeclareMathOperator{\alg}{alg}
		\DeclareMathOperator{\im}{im}

        \newcommand{\cspan}{\overline{\mathrm{span}}~}
\begin{document}
\title[Exactness and Fell bundles over inverse semigroups]{Exactness and Fell bundles with the approximation property over inverse semigroups}
\author[Changyuan Gao]{Changyuan Gao}
\address[Changyuan Gao]{Chern Institute of Mathematics and LPMC, Nankai University, Tianjin 300071, China.}
	\email{changyuangao@mail.nankai.edu.cn}

\author[Julian Kranz]{Julian Kranz}
\address[Julian Kranz]{Universit\"at M\"unster, Mathematisches Institut, Einsteinstr. 62, 48149 M\"unster, Germany}
\email{julian.kranz@uni-muenster.de}
\urladdr{https://sites.google.com/view/juliankranz/}



\date{\today}

\begin{abstract}
We prove that the reduced cross-sectional algebra of a Fell bundle with the approximation property over an inverse semigroup is exact if and only if the unit fiber of the Fell bundle is exact. This generalizes a recent result of the first-named author for actions of second countable locally compact Hausdorff groupoids on separable $C^*$-algebras. Along the way, we reprove some results of Kwa\'sniewski--Meyer on Fell bundle ideals. 
\end{abstract}

\maketitle

\section{Introduction}
The theory of exact $C^*$-algebras, initiated and largely developed by Eberhard Kirchberg in the 1990s, has become a cornerstone of modern operator algebras. A $C^*$-algebra $A$ is exact, if the minimal tensor product functor $A\otimes -$ preserves short exact sequences of $C^*$-algebras. 
Kirchberg's seminal work shows that exactness is equivalent to admitting a faithful representation with the completely positive approximation property \cite{Kirchberg1995a} and in the separable case to admitting an embedding into the Cuntz algebra $\mathcal O_2$ \cite{Kirchberg1995,Kirchberg2000}. This provides a deep link between algebraic, analytic and structural properties of $C^*$-algebras.

The related notion of nuclearity has an even longer history, dating back to the work of Takesaki \cite{Takesaki1964}.
Nuclearity for $C^*$-algebra $A$ is defined in terms of uniqueness of $C^*$-norms on the algebraic tensor product $A\odot B$ with any other $C^*$-algebra $B$ and has a remarkable analytic characterization in terms of the completely positive approximation property for the identity map on $A$ \cite{Choi1978,Kirchberg1977}.

Both nuclearity and exactness play a crucial role in the classification program, the study of crossed product $C^*$-algebras, and noncommutative geometry. 
In the setting of group $C^*$-algebras, these properties have dynamical characterizations in terms of the underlying group by means of amenability and uniform embeddability into Hilbert space. 
These dynamical conditions have important consequences including the Baum--Connes and Novikov conjectures \cite{Higson2001,Guentner2002}.

As a consequence, a lot of research in operator algebras has been devoted to finding necessary and sufficient conditions for nuclearity and exactness of $C^*$-algebras associated to more and more general classes of dynamical systems \cite{AnantharamanDelaroche1987,Exel1997,Exel2002,AnantharamanDelaroche2000,Takeishi2014,Lal-15,Abadie2022,Buss2022,Ozawa2021,BM-23,Kwasniewski2023,Kra-23,Gao-25,Buss2025}. 
The arguably most general framework for $C^*$-dynamical systems is that of Fell bundles over inverse semigroups which simultaneously generalizes groups, inverse semigroups, (twisted) \'etale groupoids, and actions of all of these on $C^*$-algebras \cite{FD-88,Kum-98,Sieben1998,Exe-11,BE-12,BM-17}.

A (saturated\footnote{Since we only consider saturated Fell bundles in this work, we simply call them \emph{Fell bundles}.}) Fell bundle over an inverse semigroup $S$ consists of a collection of Hilbert bimodules $(A_s)_{s\in S}$ equipped with a multiplication and involution that mimic the inverse semigroup operations. When $S$ is the bisection semigroup of an \'etale groupoid $G$, Fell bundles over $S$ generalize saturated Fell bundles over $G$ (see \cite[Theorem 6.1]{BM-17} or \cite[Theorem 7.5]{Kwasniewski2021}).%

On the dynamical side, nuclearity for cross-sectional algebras of Fell bundles is mirrored by Exel's \emph{approximation property} \cite{Exel1997,Exel2002,Kra-23,BM-23} which is an amenability-type condition generalizing amenability of the underlying group(oid). 
The approximation property together with nuclearity of the unit fiber often characterizes nuclearity of the cross-sectional algebra, notably for actions of discrete groups on $C^*$-algebras \cite{AnantharamanDelaroche1987}, Fell bundles over discrete groups \cite{Abadie2022,Buss2022}, and actions of second countable Hausdorff \'etale groupoids on separable $C^*$-algebras \cite{Kra-23}. In many other cases of interest, the former conditions are at least sufficient for nuclearity of the cross-sectional algebra, in particular for actions of locally compact groups on $C^*$-algebras \cite{Buss2022}, for separable Fell bundles over second countable Hausdorff \'etale groupoids \cite{Kra-23}, and for Fell bundles over inverse semigroups \cite{BM-23}. 

The main advantage of Buss--Mart\'inez' approach \cite{BM-23} to nuclearity using inverse semigroups rather than \'etale groupoids as in \cite{Kra-23} is that their purely algebraic arguments circumvent the measure-theoretic complications caused by Renault's disintegration theory \cite{Renault1987,Muhly2008}. This allows them to drop separability and Hausdorff assumptions and to achieve strictly greater generality. 

While it is not yet fully understood how exactness of cross-sectional algebras can be characterized dynamically (see \cite{AnantharamanDelaroche2021} for what is known in the groupoid case), the situation becomes accessible in the presence of the aproximation property: Under this assumption, exactness of the cross-sectional algebra is often equivalent to exactness of the unit fiber. This has been established for actions of amenable groups \cite{Kirchberg1994}, for actions of amenable groupoids \cite{Lal-15}, for group actions on $C^*$-algebras \cite{Bedos2015,McKee2021,Buss2022}, for actions of second countable Hausdorff \'etale groupoids on separable $C^*$-algebras \cite{Gao-25} and for actions of inverse semigroups on $C^*$-algebras by partial $*$-isomorphisms \cite{Gao}.

In this paper, we prove an analogous result in the generality of Fell bundles over inverse semigroups:
\begin{introthm}
Let $\mathcal A$ be a Fell bundle over a unital inverse semigroup $S$ with the approximation property. Then the reduced cross-sectional algebra $C^*_r(\mathcal A)$ is exact if and only if its unit fiber $A_1$ is exact.
\end{introthm}
As an application, we remove the Hausdorff and separability assumptions from \cite{Gao-25} in the same spirit as Buss--Mart\'inez \cite{BM-23} generalizes \cite{Kra-23} and prove the following corollary which generalizes \cite{Gao-25,Gao}:
\begin{introcor}
Let $G$ be an \'etale groupoid with Hausdorff unit space and $\mathcal A$ a saturated Fell bundle over $G$ with the approximation property. Then $C^*_r(\mathcal A)$ is exact if and only if $C_0(\mathcal A^{(0)})$ is exact.
\end{introcor} 

Our proof uses several results from Kwa\'sniewsi--Meyer's work on the ideal structure of Fell bundles, see \cite{Kwasniewski2023}. Since their arguments are quite brief, we include detailed proofs for the convenience of the reader. We do however emphasize that the results in \autoref{sec-ideals} are either explicitly or implicitly contained in \cite[Section 4]{Kwasniewski2023}.

Throughout this paper, we use $\odot$ to denote the algebraic tensor products and $\otimes$ the minimal tensor products for $C^*$-algebras and Hilbert modules. We denote by $\operatorname{id}_X\colon X\rightarrow X$ the identity map for a set $X$ or by $\operatorname{id}$ if $X$ is understood. We write $A''$ for the enveloping von Neumann algebra of a $C^*$-algebra $A$.

\section{Preliminaries}
In this section, we recall the necessary notation and basic results concerning inverse semigroups, Fell bundles, cross-sectional $C^*$-algebras, and the approximation property.

An \emph{inverse semigroup} is a semigroup $S$ such that for each $s\in S$ there is a unique $s^*\in S$ satisfying $ss^*s=s$ and $s^*ss^*=s^*$. If there exists an element $1\in S$ such that $t1=1t=t$ for any $t\in S$, then we say that $1$ is a \emph{unit} of $S$ and that $S$ is \emph{unital}. 
Note that a unit is unique if it exists. 
\begin{conv}
In this paper, all inverse semigroups are assumed to be unital.
\end{conv}
We denote by $E(S)$ the set $\{e\in S\colon e^2=e\}$ of \emph{idempotents} of $S$. Note that for any $e,f\in E(S)$, we have $e=e^*=e^2$ and $ef=fe$. Any inverse semigroup $S$ is a partially ordered set, where $t\leq u$ if there exists $e\in E(S)$ such that $t=ue$. 
For more background on inverse semigroups, we refer the reader to \cite{Lawson-98}.

Let $A$ and $B$ be $C^*$-algebras. A \emph{Hilbert $A$-$B$-bimodule} is a left Hilbert $A$-module $H$ which is also a right Hilbert $B$-module such that the left $A$-action commutes with the right $B$-action and such that the inner products satisfy the compatibility condition 
\[_A\langle\xi,\eta\rangle\zeta=\xi\langle\eta,\zeta\rangle_B\]
for all $\xi,\eta,\zeta\in H$. 
We will abbreviate the term \emph{Hilbert $A$-$A$ bimodule} to \emph{Hilbert $A$-bimodule} or simply \emph{Hilbert bimodule} if $A$ is understood.

If $H$ is a Hilbert $A$-$B$-bimodule, we denote by $H^*$ the \emph{dual Hilbert $B$-$A$-bimodule} of $H$, 
given by the conjugate $\mathbb C$-vector space $H^*=\{\xi^*\colon \xi\in H\}$, equipped with the module structures 
$b \xi^* a\coloneqq (a^* \xi b^*)^*$ and the inner products $_B\langle{\xi}^*,{\eta}^*\rangle\coloneqq\langle\xi,\eta\rangle_B$ and $\langle{\xi}^*,{\eta}^*\rangle_A\coloneqq{_A}\langle\xi,\eta\rangle$ for all $a\in A,b\in B,\xi,\eta\in H$. 
If $H$ is a Hilbert $A$-$B$-bimodule and $K$ a Hilbert $B$-$C$-bimodule, we denote by $H\otimes_BK$ the \emph{balanced tensor product} which is defined as the separated completion (i.e. completion of the quotient by all elements with vanishing seminorm) of the algebraic tensor product $H\odot_BK$ with respect to the semi-inner product defined by 
\[\langle \xi_1\otimes \eta_1,\xi_2\otimes \eta_2\rangle_C\coloneqq \langle \eta_1,\langle \xi_1,\xi_2\rangle_B\eta_2\rangle_C\]
for all $\xi_1,\xi_2\in H,\eta_1,\eta_2\in K$. Note that $H\otimes_BK$ is a Hilbert $A$-$C$-bimodule with respect to the obvious $A$-$C$-bimodule structure and the left inner product defined by 
\[{}_{A}\langle  \xi_1\otimes \eta_1,\xi_2\otimes \eta_2\rangle \coloneqq {}_{A}\langle \xi_1 {}_{B}\langle\eta_1,\eta_2\rangle ,\xi_2\rangle\]
for all $\xi_1,\xi_2\in H,\eta_1,\eta_2\in K$.
We moreover equip any $C^*$-algebra $A$ with the Hilbert $A$-bimodule structure defined by left and right multiplication and the inner products 
	\[\langle a,b\rangle_A \coloneqq a^*b,\quad _{A}\langle a,b\rangle \coloneqq ab^*,\]
for all $a,b\in A$.
For more background on Hilbert (bi)modules, we refer the reader to \cite{Lance1995}. 

The definition of Fell bundles over inverse semigroups goes back to Exel \cite{Exe-11}. 
For this paper, we use the following equivalent definition for Fell bundles, which is referred to as an action of $S$ on a $C^*$-algebra by Hilbert bimodules in \cite{BM-17}.
\begin{defn}[{\cite[Definition~4.7]{BM-17}}]\label{defn-fell-bdl}
Let $S$ be an inverse semigroup with unit and let $A$ be a $C^*$-algebra. 
A \emph{Fell bundle over $S$ with unit fiber $A$} consists of a collection of Hilbert $A$-bimodules $\mathcal A=(A_s)_{s\in S}$ and Hilbert $A$-bimodule isomorphisms $\mu_{s,u}\colon A_s\otimes_AA_u\cong A_{su}$ for $s,u\in S$, such that
\begin{enumerate}
    \item $A_1=A$ as a Hilbert $A$-bimodule;
    \item $\mu_{s,1}\colon A_s\otimes_AA\cong A_s$ and $\mu_{1,s}\colon A\otimes_AA_s\cong A_s$ are the canonical isomorphisms defined by $\mu_{s,1}(\xi\otimes a)=\xi a$ and $\mu_{1,s}(a\otimes\xi)=a\xi$ for $s\in S$, $a\in A$ and $\xi\in A_s$;
	\item for all $s,t,r\in S$, the diagram
    \begin{displaymath}
    \xymatrix@R=1cm@C=2cm{
        A_s\otimes_AA_t\otimes_AA_r \ar[d]^{\operatorname{id}_{A_s}\otimes_A\mu_{t,r}} \ar[r]^-{\mu_{s,t}\otimes_A\operatorname{id}_{A_r}} & A_{st}\otimes_AA_r \ar[d]^-{\mu_{st,r}}\\
        A_s\otimes_AA_{tr}\ar[r]^-{\mu_{s,tr}} & A_{str}
        }
    \end{displaymath}
    commutes, and the common composition is denoted by
    \[\mu_{s,t,r}\colon A_s\otimes_AA_t\otimes_AA_r\rightarrow A_{str}.\]
\end{enumerate}
\end{defn}

\begin{thm}[{\cite[Theorem~4.8]{BM-17}}]\label{inc-star}
Let $\mathcal{A}=(A_s)_{s\in S}$ be a Fell bundle over a unital inverse semigroup $S$ with unit fiber $A$. 
Then there exist unique Hilbert $A$-bimodule embeddings $\iota_{s,t}\colon A_t\rightarrow A_s$ for $s,t\in S$ with $t\leq s$ such that
\begin{enumerate}
    \item for all $s_1,t_1,s_2,t_2\in S$ with $t_1\leq s_1$ and $t_2\leq s_2$, the diagram
    \begin{displaymath}
    \xymatrix@R=1cm@C=2cm{
        A_{t_1}\otimes_AA_{t_2} \ar[d]^{\iota_{s_1,t_1}\otimes \iota_{s_2,t_2}} \ar[r]^-{\mu_{t_1,t_2}} & A_{t_1t_2} \ar[d]^-{\iota_{s_1s_2,t_1t_2}}\\
        A_{s_1}\otimes_AA_{s_2}\ar[r]^-{\mu_{s_1,s_2}} & A_{s_1s_2}
        }
    \end{displaymath}
    commutes;
    \item $\iota_{s,s}=\operatorname{id}_{A_s}$ for every $s\in S$;
    \item $\iota_{r,s}\circ \iota_{s,t}=\iota_{r,t}$ for all $t\leq s\leq r$;
\end{enumerate}
and unique Hilbert bimodule isomorphisms $J_s\colon(A_s)^*\xrightarrow{\cong} A_{s^*}$ such that 
\[\mu_{s,s^*,s}(\xi\otimes J_s({\xi^*})\otimes \xi)=\xi\langle\xi,\xi\rangle_A={_A}\langle\xi,\xi\rangle\xi\]
for $\xi\in A_s$.
If no confusion arises, we use the notation
\begin{itemize}
  \item $A_t=\iota_{u,t}(A_t)\subset A_u$;
  \item $\xi_t^*\coloneqq J_t({\xi_t^*})$;
  \item $\xi_s \xi_t\coloneqq\mu_{s,t}(\xi_s\otimes \xi_t)$
\end{itemize}
for all $s,t,u\in S$ with $t\leq u, \xi_t\in A_t$, and $\xi_s\in A_s$.
We moreover have 
\begin{itemize}
	\item $(\xi_s \xi_t)^*=\xi_t^* \xi_s^*$ for all $s,t\in S, \xi_s\in A_s$ and $\xi_t\in A_t$;
	\item $\langle \xi_s,\eta_s\rangle_{A_1} = \xi_s^*  \eta_s$ for all $s\in S$ and $\xi_s,\eta_s\in A_s$.
\end{itemize}
\end{thm}

\begin{rem}\label{rem-saturated}
	The Fell bundles appearing in \autoref{defn-fell-bdl} are also referred to as \emph{saturated} Fell bundles in the literature. 
	Our main results also hold for more general notions of Fell bundles as long as an analogue of \autoref{inc-star} holds true. 
	There is a definition of possibly non-saturated Fell bundles where the multiplication maps $\mu_{s,u}$ are only assumed to be embeddings rather than isomorphisms and instead, the maps $\mu_{s,s^*,s}$ are assumed to be isomorphisms, see \cite[Definition 2.4]{BM-23}. 
	Contrary to what is stated in \cite[Theorem~2.8]{BM-23}, this definition does not imply an analogue of \autoref{inc-star}. 
	A counterexample is given by the inverse semigroup $S=E(S)=\{e,f,1\}$ with $e<f<1$ and the Fell bundle $\mathcal A=(A_s)_{s\in S}$ given by $A_e=A_1=\C$ and $A_f=0$. 
	As kindly pointed out to us by Alcides Buss, their definition can be fixed so that an analogue of \autoref{inc-star} does indeed hold. This will appear in future work. 
	For this article however, we restrict our attention to saturated Fell bundles and simply call them \emph{Fell bundles}. 
\end{rem}
Let $\mathcal{A}=(A_s)_{s\in S}$ be a Fell bundle over a unital inverse semigroup $S$ with unit fiber $A$. 
In order to define the full and reduced cross-sectional algebras of $\mathcal A$, we need the collection of ideals $(I_{s,t})_{s,t\in S}$ in $A$ introduced in \cite{BEM-17}. 
For an element $v\in S$, note that $v^*v\in E(S)$ is an idempotent satisfying $v^*v\leq 1$. In particular, we can identify $A_{v^*v}$ with an ideal of $A=A_1$ using the inclusions $\iota_{v^*v,1}$ from \autoref{inc-star}.
Given $s,t\in S$, let $I_{s,t}\subset A$ be the closed two-sided ideal generated by $A_{v^*v}\subset A$ for all $v\leq s,t$. Then we have $I_{1,e}=A_e$ for any $e\in E(S)$ and for every $s,t,v\in S$ with $v\leq s,t$, there exists a unique Hilbert $A$-bimodule isomorphism
\begin{equation}\label{eq-theta-def}
	\theta_{s,t}\colon A_t I_{s,t}\xrightarrow{\cong} A_s I_{s,t}
\end{equation}
extending the identity map $A_t A_{v^*v}=A_v\rightarrow A_v=A_s A_{v^*v}$.%

We recall the construction of the full cross-sectional $C^*$-algebra $C^*(\mathcal A)$ of $\mathcal A$ below. For more details, see \cite{BEM-17}, \cite{BM-23} and \cite{Exe-11}.
We denote by $C_c(\mathcal{A})\coloneqq C_c(S,\mathcal{A})$ the space of finitely supported sections $S$ to $\mathcal{A}$, that is the algebraic direct sum
$\oplus_{s\in S}A_s$. We write $a_s\delta_s\in C_c(\mathcal{A})$ as the section that takes the value $a_s\in A_s$ at $s$ and $0$ everywhere else. Notice that $C_c(\mathcal{A})$ is a $*$-algebra with respect to the multiplication and the involution given by
\begin{equation}
(\xi_s\delta_s) (\xi_t\delta_t)\coloneqq(\xi_s\xi_t)\delta_{st}\qquad\text{ and } \qquad (\xi_s\delta_s)^*\coloneqq\xi_s^*\delta_{s^*}.\notag
\end{equation}
\begin{defn}[\cite{BEM-17}]\label{defn-full-alg}
The \emph{full cross-sectional $C^*$-algebra} $C^*(\mathcal{A})$ is the enveloping $C^*$-algebra of the quotient $*$-algebra $\mathbb{C}_{\alg}(\mathcal{A})\coloneqq C_c(\mathcal{A})/\mathcal{J}_{\mathcal{A}}$, where $\mathcal{J}_{\mathcal{A}}\subset C_c(\mathcal A)$ is the subspace (which is also a $*$-ideal) generated by $\theta_{t,u}(a)\delta_t-a\delta_u$ for $t,u\in S$ and $a\in A_u I_{t,u}$.
\end{defn}

As remarked in \cite[Remark~4.7]{BEM-17}, the maps $A_s\rightarrow C^*(\mathcal{A}),\quad\xi_s\mapsto \xi_s\delta_s+\mathcal{J}_{\mathcal{A}}$ for $s\in S$ are isometric and thus injective. In particular, $A_1\rightarrow C^*(\mathcal{A})$ is an injective $*$-homomorphism. 
Since this does not cause any ambiguity, we will avoid notational clutter and abbreviate the elements $\xi_s\delta_s+\mathcal J_A\in C^*(\mathcal A)$ as $\xi_s\delta_s$ throughout the paper.
\begin{defn}[{\cite[Definition~2.7]{BEM-17}}]
Let $\mathcal{A}=(A_s)_{s\in S}$ be a Fell bundle over $S$ and $H$ a Hilbert space. A \emph{representation} $\pi\colon \mathcal A\to \mathcal B(H)$ is a family $\pi=(\pi_s)_{s\in S}$ of linear maps $\pi_s\colon A_s\rightarrow \mathcal B(H)$ such that
\begin{equation}
\pi_s(\xi_s)\pi_t(\xi_t)=\pi_{st}(\xi_s\xi_t),\quad \pi_s(\xi_s)^*=\pi_{s^*}(\xi_s^*)\quad\text{and}\quad\pi_s(\xi_s)=\pi_u(\xi_s),\notag
\end{equation}
whenever $s,t,u\in S$ with $s\leq u$, $\xi_s\in A_s$ and $\xi_t\in A_t$.
We say that a representation $\pi=(\pi_s)_{s\in S}$ is \emph{non-degenerate} if $\pi_1$ is a non-degenerate representation of $A_1$.
\end{defn}

The following universal property of the full cross-sectional $C^*$-algebra is frequently utilized in this paper.
\begin{prop}[{\cite[Proposition~2.9]{BEM-17}}]\label{univ-property}
Let $\mathcal{A}=(A_s)_{s\in S}$ be a Fell bundle over $S$. Then $C^*(\mathcal{A})$ is universal for representations of $\mathcal{A}$ in the sense that for any Hilbert space $H$, there is a bijection 
\[\{ \pi=(\pi_s)_{s\in S}\colon \mathcal A\to \mathcal B(H)\}\xrightarrow{\cong}\{\tilde \pi\colon C^*(\mathcal A)\to \mathcal B(H) \},\quad \pi \mapsto \tilde \pi\]
between the representations of $\mathcal A$ on $H$ and the representations of $C^*(\mathcal A)$ on $H$ given by $\pi_s(a_s)=\tilde\pi(a_s\delta_s)$ for all $s\in S$ and $a_s\in A_s$. Moreover, $\pi$ is non-degenerate if and only if $\tilde \pi$ is non-degenerate.
\end{prop}
\begin{rem}\label{euqal-norm}
It follows from \autoref{univ-property} that for any $x\in C^*(\mathcal{A})$, we have
\[\|x\|_{C^*(\mathcal{A})}=\sup\{\|\tilde \pi(x)\|\colon \pi \text{ is any (non-degenerate) representation of } \mathcal{A}\}.\]
\end{rem}

We briefly recall the definition of the reduced cross-sectional algebra $C^*_r(\mathcal A)$ of $\mathcal A$. 
For the proofs and more details, we refer the reader to \cite{BEM-17} and \cite{BM-23}. 
We denote by $\mathcal A''=(A_s'')_{s\in S}$ the \emph{von Neumann enveloping Fell bundle} of $\mathcal A$ whose fibers are given by the Banach space biduals of the fibers of $\mathcal A$ and whose structure maps are ultraweakly continuous extensions of the structure maps of $\mathcal A$. 
There is a faithful weak conditional expectation $P\colon \C_{\mathrm{alg}}(\mathcal A'')\to A_1''$ satisfying $P(a_s\delta_s)= \theta_{1,s}''(a_s 1_s)$ where $1_s$ is the unit of $I_{1,s}''$ and $\theta_{1,s}''\colon A_s'' I_{1,s}''\to A_1''  I_{1,s}''\subset A_1''$ is the ultraweakly continuous extension of $\theta_{1,s}$ from \eqref{eq-theta-def}. 
We warn the reader the $P(\C_{\mathrm{alg}}(\mathcal A))$ is \emph{not} necessarily contained in $A_1\subset A_1''$. This failure is closely related to the possible non-Hausdorffness of the groupoid $\hat A\rtimes S$ associated to $\mathcal A$, see \cite[Remark 3.16]{Kwasniewski2023}.
We denote by $\ell^2(\mathcal A'')$ the completion of $\C_{\mathrm{alg}}(\mathcal A'')$ with respect to the $A_1''$-valued inner product $\langle \xi,\eta \rangle = P(\xi^*\eta)$. 
The \emph{reduced cross-sectional algebra} $C^*_r(\mathcal A)$ is defined as the closure of the image of the \emph{left-regular representation}
\[\Lambda\colon \C_{\textrm{alg}}(\mathcal A)\to \mathcal L(\ell^2(\mathcal A'')),\quad \Lambda(a)\xi\coloneqq a \xi\]
where we let $\C_{\textrm{alg}}(\mathcal A)\subset \C_{\textrm{alg}}(\mathcal A'')$ act on $\ell^2(\mathcal A'')$ by left multiplication. 

We end this section by recalling the definition of the approximation property from \cite{BM-23}. 
Let $\mathcal{A}=(A_s)_{s\in S}$ be a Fell bundle with unit fiber $A=A_1$. We denote by $\Gamma_c(S,\mathcal{A})$ the set of functions $\xi\colon S\rightarrow A_1$ such that $\xi(s)\in A_{ss^*}$ for every $s\in S$. 
As above, we denote by $I_{s,t}\subset A_1$ the closed two-sided ideal generated by $A_{v^*v}\subset A_1$ with $v\leq s,t$ and denote by $1_s$ the unit of $I_{s,1}''\subset A_1''$.
\begin{defn}[{\cite[Definition 4.4]{BM-23}}]
A Fell bundle $\mathcal{A}=(A_s)_{s\in S}$ over a unital inverse semigroup $S$ has the \emph{approximation property} if there is a net $(\xi_i)_{i\in I}$ in $\Gamma_c(S,\mathcal{A})$ such that
\begin{enumerate}
\item $\sup_{i\in I}||\sum_{p,t\in S}\xi_i(p)^*\xi_i(t)1_{pt^*}||_{A_1''}<\infty$;
\item for every $s\in S$ and $a_s\in A_s$, we have
   \begin{equation}
       \lim_i\left\|\sum_{p,t\in S}1_{p(st)^*}\xi_i(p)^*a_s\xi_i(t)- a_s\right\|_{A_s''}=0. \notag
   \end{equation}
\end{enumerate}
\end{defn}
The only features of the approximation property that we need in this article are the fact that it implies isomorphism of the full and reduced cross-sectional algebras \cite[Theorem 6.2]{BM-23} and the fact that it passes to tensor products \cite[Proposition 5.15]{BM-23}. 

\section{Ideals in Fell bundles}\label{sec-ideals}
Throughout this section, we fix a unital inverse semigroup $S$. 
We introduce the notion of short exact sequences of Fell bundles over $S$ and obtain a corresponding short exact sequence of full $C^*$-algebras. Based on this, we prove the main result. We acknowledge that the notion of Fell bundle ideal defined in this paper is equivalent to the notion of invariant ideal introduced by Kwa\'sniewski and Meyer in \cite{Kwasniewski2023}. This equivalence implies that most of the results (\autoref{quotient-bund-1}, \autoref{quotient-bund-2}, \autoref{lemextend} and \autoref{exactnessofuniversal}) presented here are, in fact, already contained in \cite[Section 4]{Kwasniewski2023}. Nevertheless, we have retained the complete proofs of these results, particularly to provide details that were omitted or only sketched in the original proof.

\begin{defn}\label{idealdef}
	Let $\mathcal A=(A_s)_{s\in S}$ be a Fell bundle over $S$ with unit fiber $A=A_1$. An \emph{ideal} $\mathcal I\subset \mathcal A$ consists of a collection of Hilbert $A$-subbimodules
$(I_s\subset A_s)_{s\in S}$ such that
\begin{enumerate}
  \item\label{Item-ideal-1} $I_s A_t + A_s  I_t\subset I_{st}$ for all $s,t\in S$;		
  \item\label{Item-ideal-2} $J_s(I_s^*)=I_{s^*}$ for all $s\in S$;
  \item\label{Item-ideal-3} $\iota_{t,s}(I_s)\subset I_t$ for all $s,t\in S$ with $s\leq t$.
\end{enumerate}
Here, $J_s$ and $\iota_{t,s}$ are as in \autoref{inc-star}.
\end{defn}
\begin{rem}
    Suppose that $\mathcal I=(I_s)_{s\in S}$ is an ideal of a Fell bundle $\mathcal A=(A_s)_{s\in S}$.
    \begin{enumerate}
      \item By \autoref{idealdef}~\ref{Item-ideal-1}, $I_e$ is a closed two-sided ideal in $A_e$ for every $e\in E(S)$;
	  \item Since we have $\langle \xi,\eta\rangle_{A_1}=\mu_{s^*,s}(\xi^*\otimes \eta)\in I_{s^*s}\subset I_1$ for $s\in S$ and $\xi,\eta\in I_s$, and since $I_1\cdot I_s\cdot I_1\subset I_s$, it follows that $I_s$ is a Hilbert $I_1$-bimodule with respect to the inherited operations from $A_s$.
 \item For every $s,t\in S$, the map $\mu_{\mathcal{I}_{s,t}}\colon I_s\odot I_t\rightarrow A_{st},\quad \xi_s\otimes \xi_t\mapsto \mu_{s,t}(\xi_s\otimes \xi_t)$ induces a well-defined map $\mu_{\mathcal{I}_{s,t}}\colon I_s\otimes_{I_1} I_t\rightarrow I_{st}$. Moreover, it preserves module actions and inner products. Surjectivity of $\mu_{\mathcal{I}_{s,t}}$ follows from the Cohen--Hewitt factorization theorem, surjectivity of $\mu_{s,t}\colon A_s\otimes_A A_t\to A_{st}$ and the ideal property for $\mathcal{I}$ which together imply
          \begin{align*}
            I_{st}&=\cspan I_1 I_{st} I_1\subset \cspan I_1 A_{st} I_1=\cspan I_1 A_s A_t I_1 \\
            &\subset \cspan I_s I_t\subset I_{st}.
          \end{align*}
       Hence, $\mu_{\mathcal{I}_{s,t}}$ is a Hilbert bimodule isomorphism for every $s,t\in S$. 
       \item It follows from the above, that $\mathcal I=(I_s)_{s\in S}$ is a Fell bundle over $S$ with the unit fiber $I=I_1$ and multiplication $(\mu_{\mathcal{I}_{s,t}})_{s,t\in S}$. 
    \end{enumerate}
\end{rem}

\begin{defn}[{\cite[Definition 4.8]{Kwasniewski2023}}]
	Let $\mathcal A=(A_s)_{s\in S}$ be a Fell bundle over an inverse semigroup $S$ with unit fiber $A$. 
	A closed two-sided ideal $I\subset A$ is called \emph{$\mathcal A$-invariant} if we have $I  A_s=A_s   I$ for all $s\in S$. 
\end{defn}
The following result shows that Fell bundle ideals are equivalent to invariant ideals. We refer to \cite[Proposition~6.19]{Kwasniewski2020} for further characterizations of invariant ideals. 
\begin{prop}
	Let $\mathcal A=(A_s)_{s\in S}$ be a Fell bundle over an inverse semigroup $S$ with unit fiber $A=A_1$.
	Then there is a one-to-one correspondence between ideals $\mathcal I=(I_s)_{s\in S}$ of $\mathcal A$ and $\mathcal A$-invariant ideals $I$ of $A$ given by $\mathcal I\mapsto I_1$ and $I\mapsto \mathcal A_I\coloneqq (I A_s)_{s\in S}$. 
\end{prop}
\begin{proof}
	Let $\mathcal I= (I_s)_{s\in S}$ be an ideal in $\mathcal A$. 
	Then we have $I_s = I_s  I_1 \subset A_s  I_1 \subset I_s$ and similarly $I_s= I_1 A_s $ so that $I_1$ is invariant. 
	Moreover, each $I_s$ only depends on $I_1$, so that $\mathcal I\mapsto I_1$ is injective. 
	To see that $\mathcal I\mapsto I_1$ is surjective, let $I\subset A$ be an $\mathcal A$-invariant ideal and define $I_s\coloneqq I  A_s$ for all $s\in S$. We need to show that $\mathcal I= (I  A_s)_{s\in S}$ is an ideal in $\mathcal A$. 
	Since $I$ is an ideal in $A$, it is clear that $I A_s\subset A_s$ is a Hilbert submodule. 
	Moreover, conditions \eqref{Item-ideal-1} - \eqref{Item-ideal-3} of \autoref{idealdef} are readily checked:
	\begin{enumerate}
		\item For $s,t \in S$, we have $(I  A_s) A_t = I  (A_s A_t)\subset I A_{st}$ and $A_s  (I A_t)= A_s  (A_t I) = (A_s A_t) I \subset A_{st} I =I A_{st}$;
		\item For $s\in S$, we have $(I  A_s)^*= A_s^*  I^* = A_{s^*} I = I A_{s^*}$;
		\item For $s,t \in S$ with $s\leq t$, we have $I A_s \subset I A_t$. 
	\end{enumerate}
This finishes the proof.
\end{proof}

Given an ideal $\mathcal{I}$ of a Fell bundle $\mathcal{A}$ over $S$, we next intend to construct the \emph{quotient Fell bundle} $\mathcal A/\mathcal I$. The next two results are remarked after \cite[Lemma 4.9]{Kwasniewski2023} and detailed proofs are omitted there. Here, we provide a complete proof based on the notion of Fell bundle ideals.
\begin{prop}\label{quotient-bund-1}
Let $\mathcal A=(A_s)_{s\in S}$ be a Fell bundle over an inverse semigroup $S$ with unit fiber $A= A_1$ and $\mathcal I\subset \mathcal A$ be an ideal with $I=I_1$. Then the quotient $A_s/I_s$ is a Hilbert $A/I$-bimodule with respect to the operations
		\[\langle \xi+ I_s,\eta+ I_s\rangle_{A/I} \coloneqq\langle \xi,\eta\rangle_A+I,\quad \xi,\eta\in  A_s,\]
		\[{}_{A/I}\langle \xi+ I_s,\eta+ I_s\rangle \coloneqq {}_A\langle \xi,\eta\rangle+I,\quad \xi,\eta\in  A_s,\]
		\[(\xi+ I_s)  (a+I)\coloneqq \xi  a +  I_s,\quad \xi\in  A_s,a\in A,\]
		\[(a+I)  (\xi+ I_s)\coloneqq a  \xi+ I_s,\quad \xi\in  A_s,a\in A.\]
for every $s\in S$. 
The Hilbert module norm of $A_s/I_s$ is equal to the Banach space quotient norm. 
Moreover, there is a canonical right Hilbert $A/I$-module isomorphism $A_s/I_s \cong A_s \otimes_A A/I$ and a left Hilbert $A/I$-module isomorphism $A_s/I_s\cong A/I\otimes_A A_s$. 
\end{prop}
\begin{proof}
It follows from \autoref{inc-star} and $I_s  A_t+A_s  I_t\subset I_{st}$ for any $s,t\in S$ that the above operations are well-defined.
It is straightforward to check that $A_s/I_s$ is a right \emph{semi-}inner product $A/I$-module. 
We check that the \emph{semi-}inner product is non-degenerate (i.e. that $A_s/I_s$ is an inner product $A/I$-module). 
Let $\xi\in A_s$ satisfy
\begin{equation}
\langle \xi+I_s, \xi+I_s\rangle_{A/I}=0\in A/I,\notag
\end{equation}
that is to say $\langle \xi, \xi\rangle_A\in I$. 
Since there exists a unique $\eta\in A_s$ satisfying $\xi=\eta \langle\eta,\eta\rangle_A$ (see \cite[Proposition 2.31]{RW-98}), it follows from
\begin{equation}
\langle \xi, \xi\rangle_A=\langle\eta \langle\eta,\eta\rangle_A,\eta \langle\eta,\eta\rangle_A\rangle=\langle\eta,\eta\rangle_A^3\in I\notag
\end{equation}
that $\langle\eta,\eta\rangle_A\in I$. Hence $\xi=\eta \langle\eta,\eta\rangle_A\in A_s  I\subset I_s$.
Thus, $A_s/I_s$ is a right inner product $A/I$-module. Similarly, $A_s/I_s$ is a left inner product $A/I$-module.

To see that $A_s/I_s$ is already complete (i.e. a right Hilbert $A/I$-module), we prove that the Hilbert module norm is equal to the Banach space quotient norm. 
Since $I_e, A_e$ and $I$ are closed two-sided ideals in $A$ for every $e\in E(S)$, we have 
\[A_e  I\subset I_e\subset A_e\cap I=A_e  I.\]
Hence, we have the canonical inclusion
\begin{equation}\label{eq-quotient-inclusion}
A_{s^*s}/I_{s^*s}=A_{s^*s}/(A_{s^*s}\cap I)\cong (A_{s^*s}+I)/I\subset A/I.
\end{equation}
In addition, we have 
\begin{equation}\label{eq-inner-prod-s}
\langle \xi,\xi\rangle_A=\mu_{s^*,s}(\xi^*\otimes \xi)\in A_{s^*s},\quad \text{ for all }\xi\in A_s.
\end{equation}
It follows from \eqref{eq-quotient-inclusion} and \eqref{eq-inner-prod-s} that 
\[\|\langle \xi,\xi\rangle_A+I_{s^*s}\|_{A_{s^*s}/I_{s^*s}}=\|\langle \xi,\xi\rangle_A+I\|_{A/I}.\]
Hence, we have
		\begin{align*}
			\|\langle \xi+I_s,\xi+I_s\rangle_{A/I} \|&=\|\langle \xi,\xi\rangle_A+I_{s^*s}\|_{A_{s^*s}/ I_{s^*s}}\\
			&=\inf_{x\in I_{s^*s}}\|\langle \xi,\xi\rangle_A +x\|\\
			&\leq \inf_{\eta\in I_s}\|\langle \xi,\xi\rangle_A +\underbrace{\langle \eta,\xi\rangle_A + \langle \xi,\eta\rangle_A + \langle \eta,\eta\rangle_A}_{\in I_{s^*s}} \|\\
			&=\inf_{\eta\in I_s}\|\langle \xi+\eta,\xi+\eta\rangle_A \|\\
			&=\inf_{\eta\in I_s}\|\xi+\eta\|_{A_s}^2\\
			&=\|\xi+ I_s\|^2_{A_s/I_s}.
		\end{align*}
For the converse direction, let $(e_\lambda)_\lambda$ be an approximate unit of $I_{s^*s}$. Note that 
\[\|a+I_{s^*s}\|=\lim_{\lambda}\|a-e_{\lambda}a\|\]
for each $a\in A_{s^*s}$ (see \cite[Theorem 3.1.3]{Murphy2014}). 
By substituting $\eta = -\xi e_\lambda\in I_s$, we get
		\begin{align*}
			\|\xi+I_s\|^2_{A_s/I_s}&=\inf_{\eta\in I_s}\|\langle \xi,\xi\rangle_A +\langle \eta,\xi\rangle_A + \langle \xi,\eta\rangle_A + \langle \eta,\eta\rangle_A \|\\
			&\leq \liminf_{\lambda}\|\langle \xi,\xi\rangle_A -\langle \xi e_\lambda,\xi\rangle_A - \langle \xi,\xi e_\lambda\rangle_A + \langle \xi e_\lambda,\xi e_\lambda\rangle_A \|\\
			&= \liminf_\lambda \|(1-e_\lambda)\langle\xi,\xi\rangle_A (1-e_\lambda)\|\\
			&\leq \liminf_{\lambda}\|(1-e_\lambda)\langle \xi,\xi\rangle_A \|\\
			&=\|\langle \xi,\xi\rangle_A+I_{s^*s} \|_{A_{s^*s}/I_{s^*s}}\\
            &=\|\langle \xi+I_s,\xi+I_s\rangle_{A/I} \|.
		\end{align*}
Hence, $A_s/I_s$ is a right Hilbert $A/I$-module. Similarly, $A_s/I_s$ is also a left Hilbert $A/I$-module. 

For any $\xi,\eta,\zeta\in A_s$, it follows from ${_A}\langle\xi,\eta\rangle \zeta=\xi \langle\eta,\zeta\rangle_A$ that
$${_{A/I}}\langle\xi+I_s,\eta+I_s\rangle (\zeta+I_s)=(\xi+I_s) \langle\eta+I_s,\zeta+I_s\rangle_{A/I}.$$
Hence, $A_s/I_s$ is a Hilbert $A/I$-bimodule.

For the last claim, note that the map
	\[ \Phi_s\colon A_s\odot_A A/I\to A_s/I_s,\quad \xi\otimes (a+I) \mapsto (\xi+I_s) (a+I)\]
is right $A/I$-linear and surjective by the Cohen--Hewitt factorization the\-o\-rem. 
Since we have 
\begin{align*}
	&\langle \Phi_s(\xi_1\otimes (a_1+I)),\Phi_s(\xi_1\otimes (a_1+I))\rangle_{A/I}\\
	&= \langle \xi_1a_1+I_s,\xi_2a_2+I_s\rangle_{A/I}\\
	&=(a_1+I)^*\langle \xi_1+I_s,\xi_2+I_s\rangle_{A/I}(a_2+I)\\
	&= (a_1+I)^*(\langle \xi_1,\xi_2\rangle_{A}+I)(a_2+I)\\
	&= \langle \xi_1 \otimes (a_1+I),\xi_2 \otimes (a_2+I)\rangle_{A/I},
\end{align*}
for all $\xi_1,\xi_2\in A_s$ and $a_1,a_2\in A$, the map $\Phi_s$ is isometric and thus descends to an isomorphism $A_s\otimes_A A/I\cong A_s/I_s$. 
The isomorphism $A_s/I_s\cong A/I\otimes_A A_s$ is obtained analogously.
\end{proof}
\begin{lem}\label{quotient-bund-2}
	Suppose that $\mathcal{I}$ is an ideal with unit fiber $I=I_1$ in a Fell bundle $\mathcal{A}$ over $S$ with unit fiber $A=A_1$. Then the bundle $\mathcal{A}/\mathcal{I}\coloneqq(A_s/I_s)_{s\in S}$ is a Fell bundle over $S$ with the unit fiber $A/I$ and  the multiplication $\mu_{s,t}^{\mathcal{I}}\colon A_s/I_s\otimes_{A/I} A_t/I_t\to A_{st}/I_{st}$ defined by $\mu_{s,t}^{\mathcal{I}}((\xi+I_s)\otimes(\eta+ I_t))=\mu_{s,t}(\xi\otimes\eta)+I_{st}$.
\end{lem}
\begin{proof}
For any $s,t\in S$, the map
\begin{equation}
  \tilde\mu_{s,t}\colon A_s/I_s\odot_{A/I} A_t/I_t\rightarrow A_{st}/I_{st},\quad(\xi_s+I_s)\otimes(\xi_t+I_t)\mapsto \mu_{s,t}(\xi_s\otimes \xi_t)+I_{st}\notag
\end{equation}
is well-defined by $I_s  A_t+A_s I_t\subset I_{st}$, $A/I$-bilinear by $A$-bilinearity of $\mu_{s,t}$ and surjective by surjectivity of $\mu_{s,t}$. 
Next we check that $\tilde\mu_{s,t}$ preserves inner products for all $s,t\in S$. This can be checked on elementary tensors. For all $\xi,\xi'\in A_s$ and $\eta,\eta'\in A_t$, we have
		\begin{align*}
			&\langle \tilde \mu_{s,t}((\xi+ I_s)\otimes (\eta+ I_t)),\tilde \mu_{s,t}((\xi'+ I_s)\otimes (\eta'+ I_t))\rangle_{A/I} \\
			&=  \langle  \mu_{s,t}(\xi\otimes  \eta), \mu_{s,t}(\xi'\otimes \eta')\rangle_{A} +I\\
			&=\langle \xi\otimes \eta,\xi'\otimes \eta'\rangle_{A}+I\\
			&= \langle \eta,\langle \xi,\xi'\rangle_A \eta'\rangle_A+I\\
			&= \langle \eta+ I_s,\langle \xi+ I_s,\xi'+ I_t\rangle_{A/I} (\eta'+ I_t)\rangle_{A/I}\\
			&= \langle (\xi+ I_s)\otimes (\eta+ I_t),(\xi'+ I_s)\otimes (\eta'+ I_t)\rangle_{A/I}.
		\end{align*}
Similarly, $\tilde\mu_{s,t}$ preserves the left inner product. 
Hence, $\tilde\mu_{s,t}$ induces a Hilbert bimodule isomorphism $\mu_{s,t}^{\mathcal{I}}\colon A_s/I_s\otimes_{A/I} A_t/I_t\to A_{st}/I_{st}$. 
Since $\mathcal{A}$ is a Fell bundle over $S$ with the multiplication $\{\mu_{s,t}\}_{s,t\in S}$, it is straightforward to check that $\mathcal{A}/\mathcal{I}$ is a Fell bundle over $S$ with the multiplication $\{\mu^{\mathcal{I}}_{s,t}\}_{s,t\in S}$.
\end{proof}

The following definition is equivalent to \cite[Definition~4.1]{Kwasniewski2023} which can be seen by the combination of \autoref{indeced-homo} and the proof of \cite[Proposition~4.2]{Kwasniewski2023}.
\begin{defn}\label{defn-hom}
Let $\mathcal{A}=(A_s)_{s\in S}$ and $\mathcal{B}=(B_s)_{s\in S}$ be two Fell bundles over $S$. We say $\varphi=(\varphi_s)_{s\in S}\colon \mathcal A\to \mathcal B$ is a \emph{Fell bundle homomorphism} from $\mathcal{A}$ to $\mathcal{B}$ if
\begin{enumerate}
\item $\varphi_s\colon A_s\rightarrow B_s$ is a linear map for all $s\in S$;
\item $\varphi_{st}(\xi_s \eta_t)=\varphi_s(\xi_s) \varphi_t(\eta_t)$ for all $s,t\in S$ and $\xi_s\in A_s$, $\eta_t\in A_t$;
\item $\varphi_{s^*}(\xi_s^*)=\varphi_s(\xi_s)^*$ for all $s\in S$ and $\xi_s\in A_s$; \label{item-hom-star}
\item $\varphi_s(\xi_s)=\varphi_t(\xi_s)$ for all $s,t\in S$ with $s\leq t$ and $\xi_s\in A_s$.
\end{enumerate}
In addition, we say that $\varphi$ is injective/surjective/an isomorphism if $\varphi_s$ is injective/surjective/an isomorphism for every $s\in S$.
\end{defn}
Note that for every $s\in S$, the map $(\xi_s+I_s)^*\mapsto \xi_s^*+I_{s^*}$ for $\xi_s\in A_s$ implements the canonical isomorphism $J_s\colon (A_s/I_s)^*\cong A_{s^*}/I_{s^*}$. This is used to check condition \eqref{item-hom-star} of \autoref{defn-hom} in the following example.
\begin{eg}
Suppose that $\mathcal{I}=(I_s)_{s\in S}$ is an ideal of a Fell bundle $\mathcal{A}=(A_s)_{s\in S}$ over $S$. Then the inclusion map $\iota=(\iota_s\colon I_s\to A_s)_{s\in S}\colon \mathcal I\to \mathcal A$ and the quotient map $q=(q_s\colon A_s\to A_s/I_s)_{s\in S}\colon \mathcal A\to \mathcal A/\mathcal I$ are Fell bundle homomorphisms, where $\iota_s\colon I_s\hookrightarrow A_s$ and $q_s\colon A_s\twoheadrightarrow A_s/I_s$ are the canonical inclusion and quotient maps.
Note that $\iota$ is injective whereas $q$ is surjective. 
\end{eg}
\begin{prop}\label{indeced-homo}
Let $\mathcal{A}=(A_s)_{s\in S}$ and $\mathcal{B}=(B_s)_{s\in S}$ be two Fell bundles over $S$. Then any Fell bundle homomorphism $\varphi\colon\mathcal{A}\rightarrow \mathcal{B}$ induces a $*$-ho\-mo\-mor\-phism 
\[C^*(\varphi)\colon C^*(\mathcal{A})\rightarrow C^*(\mathcal{B}),\quad C^*(\varphi)(\xi_s \delta_s)=\varphi_s(\xi_s)\delta_s.\]
\end{prop}
\begin{proof}
Suppose that $C^*(\mathcal{B})\subset \mathcal B(H)$ for some Hilbert space $H$.
For $s\in S$, we define $\pi_s$ as the composition
\begin{equation*}
\pi_s\colon A_s\stackrel{\varphi_s}{\longrightarrow} B_s\hookrightarrow C^*(\mathcal{B})\subset \mathcal B(H).
\end{equation*}
Then $\pi_s(\xi_s)=\varphi_s(\xi_s)\delta_s$ for every $\xi_s\in A_s$. It is straightforward to check that $\pi=(\pi_s)_{s\in S}$ is a representation of $\mathcal{A}$ on $H$. By \autoref{univ-property}, there is a homomorphism
$C^*(\varphi)\colon C^*(\mathcal{A})\rightarrow C^*(\mathcal{B})$ such that $C^*(\varphi)(\xi_s\delta_s)=\varphi_s(\xi_s)\delta_s$.
\end{proof}

\begin{defn}
Let $S$ be a unital inverse semigroup. A sequence
\[0\longrightarrow \mathcal{I}\stackrel{\iota}{\longrightarrow}  \mathcal{A}\stackrel{\rho}{\longrightarrow} \mathcal{B}\longrightarrow 0\]
of Fell bundles over $S$ with Fell bundle homomorphisms is called \emph{exact} if $\iota$ is an ideal inclusion and $\rho$ is a surjective Fell bundle homomorphism such that $\ker (\rho_s)=\im (\iota_s)$ for all $s\in S$.
\end{defn}
\begin{rem}\label{induced-fell-bun-iso}
In the above definition, the Fell bundle homomorphism $\rho$ induces a canonical Fell bundle isomorphism $\tilde{\rho}\colon\mathcal{A}/\mathcal{I}\cong \mathcal{B}$ satisfying $\rho=\tilde{\rho}\circ q$, where $q\colon\mathcal{A}\rightarrow \mathcal{A}/\mathcal{I}$ is the quotient map. Moreover, we have $C^*(\rho)=C^*(\tilde{\rho})\circ C^*(q)$.
\end{rem}
Suppose that $\mathcal A=(A_s)_{s\in S}$ is a Fell bundle over a unital inverse semigroup $S$ and that $\pi=(\pi_s)_{s\in S}\colon\mathcal{A}\to \mathcal B(H)$ is a representation. 
If $\mathcal{I}=(I_s)_{s\in S}$ is an ideal of $\mathcal{A}$, then the restriction $\pi|_{\mathcal{I}}=(\pi_s|_{I_s})_{s\in S}$ is a representation of $\mathcal{I}$. 
Conversely, given a non-degenerate representation of the ideal $\mathcal{I}$, the following result shows that we can extend it to obtain a representation of $\mathcal{A}$. 
This is subsequently used to prove that $C^*(\mathcal{I})$ embeds into $C^*(\mathcal{A})$. Note that the following result is included in the proof of \cite[Proposition 4.15]{Kwasniewski2023}. For the convenience of reader, we provide a complete proof along with detailed supplementary explanations.

\begin{lem}\label{lemextend}
Let $\mathcal{I}=(I_s)_{s\in S}$ be an ideal of a Fell bundle $\mathcal{A}=(A_s)_{s\in S}$ over $S$. 
If $\pi=(\pi_s)_{s\in S}$ is a non-degenerate representation of $\mathcal{I}$ on a Hilbert space $H$, then $\pi$ can be extended to a representation $\tilde{\pi}=(\tilde{\pi}_s)_{s\in S}$ of $\mathcal{A}$ on $H$.
\end{lem}
\begin{proof}
Let $\pi=(\pi_s)_{s\in S}\colon \mathcal I\to \mathcal B(H)$ be a non-degenerate representation. 
We denote by $I=I_1$ the unit fiber of $\mathcal I$. 
We aim to define a representation $\tilde \pi =(\tilde{\pi}_s)_{s\in S}\colon \mathcal A\to \mathcal B(H)$ by defining 
\begin{equation}
	\tilde\pi_s(a_s)(\pi_1(b)h)\coloneqq\pi_s(a_sb)h
\end{equation}
for any $a_s\in A_s, b\in I$ and $h\in H$. 

To verify that $\tilde\pi_s(a_s)$ is well-defined for any $a_s\in A_s$, we pick an approximate unit $(e_{\lambda})_{\lambda}$ of $I$. 
Suppose that there exist $\pi_1(b_1)h_1=\pi_1(b_2)h_2 \in \pi_1(I)H$. 
Then 
\begin{align*}
\pi_s(a_sb_1)h_1&=\lim_{\lambda}\pi_s(a_se_{\lambda}b_1)h_1=\lim_{\lambda}\pi_s(a_se_{\lambda})\pi_1(b_1)h_1\\
& = \lim_{\lambda}\pi_s(a_se_{\lambda})\pi_1(b_2)h_2=\pi_s(a_sb_2)h_2.
\end{align*}
Since we have $H=\pi_1(I)H$ by the Cohen--Hewitt factorization theorem, we see that $\tilde\pi_s(a_s)\colon H\rightarrow H$ is well-defined for any $a_s\in A_s$. 
To see that $\tilde\pi_s(a_s)$ is bounded, note that 
\begin{align*}
		\|\tilde\pi_s(a_s)(\pi_1(b_1)h_1)\|&= \|\lim_\lambda \pi_{s}(a_s e_\lambda ) \pi_1(b_1)h_1\| \\
		&\leq \limsup_\lambda \|\pi_{s}( a_s e_\lambda ) \pi_1(b_1)h_1\|\\
		&\leq \limsup_\lambda \|\pi_{s}(a_s e_\lambda )\|   \| \pi_1(b_1)h_1\|\\
		&\leq \underbrace{\limsup_\lambda \|e_\lambda\|}_{\leq 1}  \|a_s\|_{A_{s}}  \| \pi_1(b_1)h_1\|.
	\end{align*}
Hence, $\tilde\pi_s(a_s)$ is a bounded linear map on $H$.

Next we check that $\tilde\pi$ is a representation of $\mathcal{A}$. Firstly, for $a_s\in A_s$, $a_t\in A_t$, $b\in I$ and $h\in H$ we have
\begin{equation*}
\tilde\pi_{st}(a_sa_t)(\pi_1(b)h)=\pi_{st}(a_sa_tb)h=\tilde \pi_s(a_s)\pi_t(a_tb)h
=\tilde\pi_{s}(a_s)\tilde\pi_{t}(a_t)\pi_1(b)h.
\end{equation*}
Hence, $\tilde\pi_{st}(a_sa_t)=\tilde\pi_{s}(a_s)\tilde\pi_{t}(a_t)$. Next, for $a_s\in A_s$, $b_1, b_2\in I$ and $h_1,h_2\in H$ we have
\begin{align*}
\langle\tilde\pi_{s}(a_s)(\pi_1(b_1)h_1),\pi_1(b_2)h_2\rangle &=\langle\pi_{s}(a_sb_1)h_1,\pi_1(b_2)h_2\rangle\\
&=\langle h_1,\pi_{s^*}(b_1^*a_s^*b_2)h_2\rangle\\
&=\langle h_1,\pi_1(b_1^*)\pi_{s^*}(a_s^*b_2)h_2\rangle\\
&=\langle \pi_1(b_1)h_1,\pi_{s^*}(a_s^*b_2)h_2\rangle\\
&=\langle \pi_1(b_1)h_1,\tilde\pi_{s^*}(a_s^*)\pi_1(b_2)h_2\rangle.
\end{align*} 
Hence, $\tilde\pi_{s^*}(a_s^*)=\tilde\pi_{s}(a_s)^*$.
Finally, for any $a_s\in A_s,b \in I,h\in H$ and $s,u\in S$ with $s\leq u$ we have
\[\tilde\pi_s(a_s)\pi_1(b)h=\pi_s(a_sb)h=\pi_u(a_sb)h=\tilde\pi_u(a_s)\pi_1(b)h.\]
Hence, $\tilde\pi_s(a_s)=\tilde\pi_u(a_s)$.
This shows that $\tilde \pi$ is a representation. It is easy to see that $\tilde \pi$ extends $\pi$. 
\end{proof}

To obtain the main result in this paper, We will use the following result which is exactly \cite[Proposition 4.15]{Kwasniewski2023}. Here, we reprove it for completeness. 
\begin{thm}\label{exactnessofuniversal}
Let
\[0\longrightarrow \mathcal{I}\stackrel{\iota}{\longrightarrow} \mathcal{A}\stackrel{q}{\longrightarrow}\mathcal B\longrightarrow 0\]
be a short exact sequence of Fell bundles. Then the induced sequence
	\[ 0\longrightarrow C^*(\mathcal I)\xrightarrow{C^*(\iota)} C^*(\mathcal A)\xrightarrow{C^*(q)} C^*(\mathcal B)\longrightarrow 0\]
of $C^*$-algebras is exact. 
\end{thm}

\begin{proof}
	In view of \autoref{induced-fell-bun-iso}, we may identify $\mathcal B$ with the quotient Fell bundle $\mathcal A/\mathcal I$ and $\iota$ and $q$ the canonical inclusion and quotient map. 
	We first show that $C^*(\iota)$ is injective. 
	For a Fell bundle $\mathcal{C}$ over $S$, we write
		\[ \mathcal \C_{\alg}(\mathcal C)\coloneqq C_c(\mathcal C)/\mathcal{J}_{\mathcal{C}}.\]
as in \autoref{defn-full-alg}. 
Since $\C_{\alg}(\mathcal I)$ embeds densely into $C^*(\mathcal{I})$, it suffices to prove that $\|C^*(\iota)(x)\|_{C^*(\mathcal{A})}\geq\|x\|_{C^*(\mathcal{I})}$ for every $x=\sum_{s\in S}a_s\delta_s+\mathcal{J}_{\mathcal{I}}\in \C_{\alg}(\mathcal I)$, where the sum has only finitely many non-zero terms. 
Using \autoref{euqal-norm} at the first and last step and \autoref{lemextend} at the second step, we have
\begin{align*}
  \|x\|_{C^*(\mathcal{I})} & =\sup\left\{\left\|\sum_{s\in S}\pi_s(a_s)\right\|\colon\pi \text{ is a (non-degenerate) representation of } \mathcal{I}\right\} \\
  & \leq \sup\left\{\left\|\sum_{s\in S}\pi_s(a_s)\right\|\colon\pi \text{ is a (non-degenerate) representation of } \mathcal{A}\right\} \\
  & = \left\|\sum_{s\in S}a_s\delta_s+\mathcal{J}_{\mathcal{\mathcal{A}}}\right\|_{C^*(\mathcal{A})},
\end{align*}
Hence, $\|C^*(\iota)(x)\|_{C^*(\mathcal{A})}=\|x\|_{C^*(\mathcal{I})}$ and $C^*(\iota)$ is injective.

	Next, we show that $\ker(C^*(q))=\im(C^*(\iota))$. It follows from $\im(\iota_s)\subset \ker(q_s)$ for every $s\in S$ that $\im (C^*(\iota))\subset \ker (C^*(q))$. Note that $\im (C^*(\iota))$ is an ideal of $C^*(\mathcal{A})$ since $\mathcal{I}$ is an ideal of $\mathcal{A}$. For the reverse inclusion, it suffices to show that the natural map
		\[Q\colon C^*(\mathcal A)/C^*(\mathcal I)\to C^*(\mathcal A/\mathcal I)\]
	induced by $C^*(q)$ is isometric, where $C^*(\mathcal I)\subset C^*(\mathcal A)$ by $C^*(\iota)$.

Let $\varphi\colon C^*(\mathcal A)/C^*(\mathcal I)\to  \mathcal B(H)$ be a faithful non-degenerate representation and
		\[P\colon C^*(\mathcal A)\to C^*(\mathcal A)/C^*(\mathcal I)\]
	be the canonical quotient map.
	By \autoref{univ-property}, 
	\[\tilde\pi\colon C^*(\mathcal A)\xrightarrow{P}C^*(\mathcal A)/C^*(\mathcal I)\xrightarrow{\varphi} \mathcal B(H)\]
	is induced by a non-degenerate representation $\pi\colon \mathcal A\to \mathcal B(H)$ so that $\tilde \pi(a_s\delta_s +\mathcal J_A)=\pi_s(a_s)$ for all $s\in S$ and $a_s\in A_s$.
	Note that for $s \in S$ and $a\in I_s$, we have $a\delta_s\in C^*(\mathcal I)\subset C^*(\mathcal A)$. In particular, we have $\pi_s(a) = \tilde \pi(a \delta_s +\mathcal J_A)=\varphi(P(a \delta_s +\mathcal J_A))=0$. Therefore, $\pi$ vanishes on $\mathcal I$ and factors through a non-degenerate representation $\rho \colon \mathcal A/\mathcal I\to \mathcal B(H)$. 
	Denote by $\tilde \rho\colon C^*(\mathcal A/\mathcal I)\to \mathcal B(H)$ the induced representation of $\rho$ by the universal property of \autoref{univ-property}.
	Then we have $\varphi = \tilde\rho\circ Q$ since for each $x\in C^*(\mathcal A)/C^*(\mathcal I)$ of the form $x=\sum_{s\in S}a_s\delta_s + C^*(\mathcal I)$ for finitely many non-zero $a_s\in A_s$, we have 
	\begin{align*}
		 \tilde \rho\circ Q\left(\sum_{s\in S}a_s \delta_s+C^*(\mathcal I)\right)&=\tilde \rho\left( \sum_{s\in S}(a_s+I_s) \delta_s\right)\\
		 &=\sum_{s\in S}\rho_s(a_s+I_s)\delta_s\\
		 &=\sum_{s\in S} \pi_s(a_s)\delta_s\\
		 &=\varphi\circ P\left(\sum_{s\in S}a_s\delta_s\right)\\
		 &=\varphi\left(\sum_{s\in S}a_s \delta_s+C^*(\mathcal I)\right).
	\end{align*}
	Since $\varphi$ is faithful and therefore isometric, we conclude that $Q$ is isometric.
\end{proof}

Now that we understand how taking full cross-sectional algebras interacts with short exact sequences of Fell bundles, we would like to understand how short exact sequences interact with taking tensor products of Fell bundles. 
Given a Fell bundle $\mathcal{A}=(A_s)_{s\in S}$ over $S$ and a $C^*$-algebra $B$, Buss and Mart\'{\i}nez \cite[Section 5]{BM-23} construct a new Fell bundle $\mathcal{A}\otimes B$ over $S$ where every fiber $(\mathcal{A}\otimes B)_s$ is defined as the minimal external tensor product $A_s\otimes B$ of $A_s$ and $B$, where $B$ is viewed as a Hilbert $B$-$B$-bimodule. For every $s,t\in S$, the map
\begin{equation}
(A_s\odot B)\odot(A_t\odot B)\rightarrow A_{st}\otimes B,\quad(a_s\otimes b_1)\otimes(a_t\otimes b_2)\mapsto a_sa_t\otimes b_1b_2\notag
\end{equation}
extends to a Hilbert bimodule isomorphism $(A_s\otimes B)\otimes_{A_1\otimes B}(A_t\otimes B)\rightarrow A_{st}\otimes B$
(see \cite[Lemma 5.2]{BM-23}). 

\begin{lem}\label{shortexactfell}%
Let $\mathcal A=(A_s)_{s\in S}$ be a Fell bundle over a unital inverse semigroup $S$ with exact unit fiber $A=A_1$. If $I$ is a closed two-sided ideal of a $C^*$-algebra $B$, then the canonical sequence 
\[0\to \mathcal A\otimes I\to \mathcal A\otimes B\to \mathcal A\otimes (B/I)\to 0\]
is exact. 
\end{lem}
\begin{proof}
It follows from injectivity of the minimal tensor product that the map $\mathcal A\otimes I\to \mathcal A\otimes B$ is an ideal inclusion. 
Moreover, it is straightforward that the canonical map $\mathcal A\otimes B\to \mathcal A\otimes (B/I)$ is surjective and that the composition $\mathcal A\otimes I\to \mathcal A\otimes B\to \mathcal A\otimes (B/I)$ is zero. 
We check that the canonical map $\psi=(\psi_s)_{s\in S}\colon (\mathcal A\otimes B)/(\mathcal A\otimes I)\to \mathcal A\otimes (B/I)$ is isometric by checking it on the image of the algebraic tensor product $A_s\odot B$ in $(A_s\otimes B)/(A_s\otimes I)$. 
Let $s\in S$ and $\sum_{i=1}^n\xi_i\otimes b_i\in A_s\odot B$. Using exactness of $A$ at the third step, we have 
\begin{align*}
  &\hspace{-2cm}\left\|\psi_s\left(\sum_{i=1}^n\xi_i\otimes b_i+A_s\otimes I\right)\right\|^2 \\
  =& \left\|\left\langle\sum_{i=1}^n\xi_i\otimes (b_i+I),\sum_{j=1}^n\xi_j\otimes (b_j+I)\right\rangle_{A\otimes B/I}\right\|\\
  =&\left\|\sum_{i,j=1}^n\langle\xi_i,\xi_j\rangle_A\otimes(b_i^*b_j+I)\right\|\\
  =&\left\|\sum_{i,j=1}^n\langle\xi_i,\xi_j\rangle_A\otimes b_i^*b_j+A\otimes I \right\|\\
  =&\left\|\left\langle\sum_{i=1}^n\xi_i\otimes b_i+A_s\otimes I,\sum_{j=1}^n\xi_j\otimes b_j+A_s\otimes I\right\rangle\right\| \\
  =&\left\|\sum_{i=1}^n\xi_i\otimes b_i+A_s\otimes I\right\|^2.%
\end{align*}
\end{proof}

\section{Proof of the main result}
\begin{thm}\label{mainthm}
Let $\mathcal{A}=(A_s)_{s\in S}$ be a Fell bundle over $S$ with the approximation property. Then $C^*_r(\mathcal A)$ is exact if and only if the unit fiber $A= A_1$ is exact.
\end{thm}
\begin{proof}
	If $C^*_r(\mathcal A)$ is exact, then $A\subset C^*_r(\mathcal A)$ is exact since exactness passes to $C^*$-subalgebras (see \cite[IV.3.4.3]{BL-06}).
	
For the converse direction, assume that $A$ is exact and let
		\[0\longrightarrow  I\longrightarrow  B \longrightarrow  B/ I\longrightarrow 0\]
be a short exact sequence of $C^*$-algebras.
It suffices to show that the sequence
\begin{equation*}\label{eq-shortexacttensor}
	0\longrightarrow C^*_r(\mathcal A)\otimes I\longrightarrow C^*_r(\mathcal A)\otimes B\longrightarrow C^*_r(\mathcal A)\otimes B/I\longrightarrow 0
\end{equation*}
is exact.

For any $C^*$-algebra $D$, it follows from \cite[Proposition 5.15]{BM-23} that $\mathcal{A}\otimes D$ has the approximation property and from \cite[Theorem 6.2]{BM-23} that the left regular representation induces an isomorphism 
\begin{equation}\label{eq-left-regular}
	C^*(\mathcal A\otimes D)\cong C^*_r(\mathcal A\otimes D).
\end{equation}
Moreover, the proof of \cite[Proposition 5.9]{BM-23} shows that the canonical $*$-ho\-mo\-mor\-phism
\begin{equation}
	\phi_0\colon\mathbb{C}_{\rm{alg}}(\mathcal{A})\odot D\rightarrow \mathbb{C}_{\rm{alg}}(\mathcal{A}\otimes D)\subset C_r^*(\mathcal{A}\otimes D),\notag
\end{equation}
given by
\begin{equation}
	\phi_0\left(\left(\sum_{s\in S}a_s\delta_s+\mathcal{J}_{\mathcal{A}}\right)\otimes d\right)=\sum_{s\in S}(a_s\otimes d) \delta_s+\mathcal{J}_{\mathcal{A}\otimes D},\notag
\end{equation}
extends to the $*$-isomorphism 
\begin{equation}\label{eq-tensor-exchange}
	C_r^*(\mathcal{A})\otimes D\cong C_r^*(\mathcal{A}\otimes D)
\end{equation}
 for any $C^*$-algebra $D$. 
By applying the natural isomorphisms \eqref{eq-left-regular} and \eqref{eq-tensor-exchange} to $D\in \{I,B,B/I\}$, we obtain a commutative diagram 
\begin{equation*}
    \xymatrix{
      0 \ar[r] & C_r^*(\mathcal A)\otimes I \ar[r]  & C_r^*(\mathcal A)\otimes B \ar[r]  & C_r^*(\mathcal A)\otimes B/I \ar[r] & 0 \\
      0 \ar[r] & C^*(\mathcal A\otimes I)\ar[r]\ar[u]^\cong & C^*(\mathcal A\otimes B) \ar[r]\ar[u]^\cong & C^*(\mathcal A\otimes B/I) \ar[r] \ar[u]^\cong & 0.
    }
\end{equation*}
The lower row of this diagram is exact by the combination of \autoref{shortexactfell} and \autoref{exactnessofuniversal}. Thus, the upper row is exact as well. 
\end{proof}
We end this paper with an application to Fell bundles over \'etale grou\-poids.
We first recall some notation.
Let $G$ be a locally compact groupoid with Hausdorff unit space $G^{(0)}$. 
A subset $A\subset G$ is called a \emph{bisection} if the range and source maps $r,s\colon A\rightarrow G^{(0)}$ are homeomorphisms onto their images. 
Note that the set $\mathrm{Bis}(G)$ of all open bisections is an inverse semigroup with the unit $G^{(0)}$. 
We say that $G$ is \emph{\'etale} if the range map $r\colon G\rightarrow G^{(0)}$ is a local homeomorphism. 
We refer the reader to \cite[Definition 2.6]{BE-12} for the definition of Fell bundles over a groupoid and to \cite[Definition 4.10]{BM-23} for the definition of the approximation property for a Fell bundle over a groupoid. 

Given a saturated Fell bundle $\mathcal{A}=(A_g)_{g\in G}$ over $G$ and a \emph{wide} inverse semigroup $S\subset \operatorname{Bis}(G)$ (see \cite[Definition 4.11]{BM-23}, for instance $S= \operatorname{Bis}(G)$), there is an associated Fell bundle $\mathcal{B}=(B_u)_{u\in S}$ over $S$, where $B_u\coloneqq C_0(\mathcal{A}_u)$ is the space of continuous sections vanishing at infinity of the restriction $\mathcal{A}_u$ of $\mathcal{A}$ to $u$ (see \cite[Example 2.9]{BE-12} or \cite[Example 3.4]{Kwasniewski2023}). 
The unit fiber of $\mathcal B$ is given by the $C^*$-algebra $B\coloneqq C_0(\mathcal A^{(0)})$ of continuous sections vanishing at infinity of the upper semicontinuous $C^*$-bundle $\mathcal A^{(0)}\coloneqq \mathcal A|_{G^{(0)}}$ over $G^{(0)}$. 
Moreover, there are canonical isomorphisms $C^*(\mathcal{A})\cong C^*(\mathcal{B})$ and $C_r^*(\mathcal{A})\cong C_r^*(\mathcal{B})$ (see \cite{BEM-17} and \cite{BM-17}).

The following result is the exactness version of \cite[Corollary 6.9(2)]{BM-23} and generalizes the main results of \cite{Gao-25, Gao}. Here, we make neither the assumption that the groupoid is Hausdorff and second-countable, nor the assumption that the Fell bundle is separable.
\begin{cor}\label{fell-bundle-gpd-ext}
	Let $G$ be an \'etale groupoid with Hausdorff unit space. Let $\mathcal A$ be a saturated Fell bundle over $G$ with the approximation property. Then $C_r^*(\mathcal A)$ is exact if and only if $C_0(\mathcal A^{(0)})$ is exact.
\end{cor}
\begin{proof}
If $C^*(\mathcal A)$ is exact, then $C_0(\mathcal A^{(0)})$ is exact since $C_0(\mathcal A^{(0)})$ embeds canonically into $C^*(\mathcal A)\cong C_r^*(\mathcal A)$.
Conversely, we assume that $C_0(\mathcal A^{(0)})$ is exact. Let $\mathcal{B}=(B_u)_{u\in S}$ be the Fell bundle over $S=\operatorname{Bis}(G)$ associated to $\mathcal{A}$. It follows from \cite[Theorem 4.16]{BM-23} that $\mathcal{B}$ has the approximation property. Since $C_0(\mathcal{A}^{(0)})=B_1$ is exact, then $C_r^*(\mathcal{A})\cong C_r^*(\mathcal{B})$ is exact by \autoref{mainthm}.
\end{proof}
\subsection*{Acknowledgements}
This work was supported by the Deutsche Forschungsgemeinschaft (DFG, German Research Foundation) under Germany's Excellence Strategy EXC 2044/2 - 3 90685587, Mathematics M\"unster: Dynamics--Geometry--Struc\-ture, and by the SFB 1442 of the DFG.
We appreciate helpful conversations with Alcides Buss, Bartosz Kwa\'sniewski, and Diego Mart\'inez. We gratefully acknowledge helpful comments by the referees.
\bibliography{references.bib}
\bibliographystyle{alphaurl}

\end{document}